\documentclass{amsart}
\usepackage{amssymb}
\usepackage{amsmath}
\usepackage{hyperref}
\newcommand{\la}{\lambda}

\newcommand{\om}{\omega}

\newcommand{\bee}{\begin{eqnarray*}}
	\newcommand{\eee}{\end{eqnarray*}}
\newcommand{\ba}{\begin{aligned}}
	\newcommand{\ea}{\end{aligned}}
\newcommand{\bp}{\begin{proof}}
	\newcommand{\ep}{\end{proof}}  
\newcommand{\br}{\begin{remark}}
	\newcommand{\er}{\end{remark}}  
\newcommand{\lb}{\label}

\newtheorem{thm}{Theorem}[section]

\newtheorem{lem}[thm]{Lemma}

\theoremstyle{definition}
\newtheorem{defn}[thm]{Definition}

\theoremstyle{remark}

\usepackage{enumitem}

\newcommand{\Field}{\mathbb{F}}

\begin{document}
	\title[A primitive normal pair with prescribed prenorm]{A primitive normal pair with prescribed prenorm}
	\author[K. Chatterjee and S.K. Tiwari]{ Kaustav Chatterjee$^{*}$ and Shailesh Kumar Tiwari}
	\address{Department of Mathematics,  Indian Institute of Technology
		Patna, BIHAR, INDIA.}
	\email{kaustav\_2121ma03@iitp.ac.in}
	\address{Department of Mathematics,  Indian Institute of Technology
		Patna, BIHAR, INDIA.}
	\email{ sktiwari@iitp.ac.in}
	\thanks{*email:kaustav0004@gmail.com}
	\thanks{ }
	\maketitle
	\begin{abstract}
	For any positive integers $q$, $n$, $m$ with $q$ being a prime power and $n \geq 5$, we establish a condition sufficient to ensure the existence of a primitive normal pair $(\epsilon,f(\epsilon))$ in $\Field_{q^{n}}$ over $\Field_{q}$ such that $\mathrm{PN}_{q^n/q}(\epsilon)=a$, where $a\in\Field_{q}$ is prescribed. Here $f={f_{1}}/{f_{2}}\in\mathbb{F}_{q^n}(x)$ is a rational function subject to some minor restrictions such that deg($f_{1}$)+deg($f_{2}$)$=m$ and $\mathrm{PN}_{q^n/q}(\epsilon)
	=\sum_{i=0}^{n-1}\Bigg(\underset{j\neq i}{\underset{0\leq j\leq n-1}{\prod_{}^{}}}\epsilon^{q^j}\Bigg)$. Finally, we conclude that for $m=3$, $n\geq 6$, and $q=7^k$ where $k\in\mathbb{N}$, such a pair will exist certainly for all $(q,n)$ except possibly $10$ choices at most.
	\end{abstract} 
	\section{Introduction}
	Let $\Field_{q^n}$ represents a finite extension of degree $n$ of the finite field $\Field_{q}$, for some prime power $q$ and positive integer $n$. There are two group structures associated to $\Field_{q^{n}}$, one is the additive group $\Field_{q^{n}}$ and another one is the multiplicative group $\mathbb{F}_{q^n}^*$. Additionally, the latter one forms a cyclic group and any generator of the same is termed as a primitive element of $\Field_{q^n}$. Thus, there exist $\phi(q^n-1)$ primitive elements in the finite field $\Field_{q^n}$, where $\phi$ is the Euler's totient function. The least degree monic irreducible polynomial over $\Field_{q}$ having a primitive root $\epsilon\in\Field_{q^{n}}^*$, is referred as primitive polynomial of the primitive element. For any $\epsilon$ $\in\Field_{q^n}$, the elements $\epsilon, \epsilon^{q}, \ldots, \epsilon^{q^{n-1}}$ are said to be the conjugates of $\epsilon$ with respect to $\Field_{q}$. Clearly, the set $\{\epsilon, \epsilon^{q}, \ldots, \epsilon^{q^{n-1}}\}$ spans a subspace of $\Field_{q^{n}}$ over $\Field_{q}$. In particular, for some $\epsilon\in\Field_{q^{n}}$, the set contained with the conjugates of $\epsilon$ forms a basis of $\Field_{q^n}$ over $\Field_q$, then it is said to be a normal basis, whereas the element is referred as a normal element. A primitive normal pair $(\epsilon,\delta) \in \mathbb{F}_{q^n}^* \times \mathbb{F}_{q^n}^*$ is characterized by the property that both $\epsilon$ and $\delta$ are primitive and normal with respect to the finite field $\mathbb{F}_{q}$. To obtain further information on primitive normal elements in finite fields, we recommend \cite{RH} to the reader. In this article, first we shall present the following definition.
	\begin{defn}
			For any $\epsilon\in\Field_{q^{n}}$, the \textit{prenorm} of $\epsilon$ over $\Field_{q}$ is denoted by  $\mathrm{PN}_{q^n/q}(\epsilon)$ and is defined by
		\begin{equation}\nonumber
			\mathrm{PN}_{q^n/q}(\epsilon)=\sum_{i=0}^{n-1}\Bigg(\underset{j\neq i}{\prod_{j=0}^{n-1}}\epsilon^{q^j}\Bigg).
		\end{equation}
			In other words, considering the products of the form $\epsilon\cdot\epsilon^q\ldots\epsilon^{q^{j-1}}\cdot\epsilon^{q^{j+1}}\ldots\epsilon^{q^{n-1}} (0\leq j \leq n-1)$ and then summing up these products, we get the prenorm of $\epsilon$ over $\Field_{q}$.
	\end{defn}
Let $\Upsilon(x)=x^n+a_{n-1}x^{n-1}+\ldots+a_{1}x+a_{0}\in\Field_{q}[x]$ be the minimum polynomial of $\epsilon$.  Further, the roots of $\Upsilon$ are $\epsilon, \epsilon^{q}, \ldots, \epsilon^{q^{n-1}}$. Thus 
	\begin{equation}\nonumber
		\begin{aligned}
			\Upsilon(x)&=x^n+a_{n-1}x^{n-1}+\ldots+a_{1}x+a_{0}\\
			&=(x-\epsilon)(x-\epsilon^q)\ldots(x-\epsilon^{q^{n-1}}),
		\end{aligned}
	\end{equation}
	and following the comparison of the coefficients we get that $\mathrm{PN}_{q^n/q}(\epsilon)=(-1)^{n-1}a_{1}$, that is, $\mathrm{PN}_{q^n/q}(\epsilon)\in\Field_{q}$. Before this article, conditions were proved to ensure the presence of a primitive pair for a primitive normal pair $(\epsilon,f(\epsilon))$, where $f(x)$ is a rational function in $\Field_{q^{n}}(x)$ with certain restrictions, along with a specified trace or norm. In this article, our aim is to identify those pairs $(q,n)$ for which the field $\Field_{q^{n}}$ contains a primitive normal pair $(\epsilon,f(\epsilon))$ over $\Field_{q}$, satisfying $\mathrm{PN}_{q^n/q}(\epsilon)=a$, for prescribed $a\in\Field_{q}$. Furthermore, the \textit{trace} of an element $\epsilon \in \Field_{q^n}$ over $\Field_{q}$, represented as $\mathrm{Tr}_{{q^n}/q}(\epsilon)$ and is defined as $\mathrm{Tr}_{{q^n}/q}(\epsilon)=\epsilon+\epsilon^{q}+\ldots+\epsilon^{q^{n-1}}$. Similarly, the \textit{norm} of an element $\epsilon \in \Field_{q^n}$ over $\Field_{q}$, denoted by $\mathrm{N}_{{q^n}/q}(\epsilon)$ and is defined as $\mathrm{N}_{{q^n}/q}(\epsilon)=\epsilon\cdot\epsilon^{q}\ldots\epsilon^{q^{n-1}}$. To proceed, we shall use of the following result.
	\begin{lem}\lb{L1.1}
		Assume that $q$, $n$ be positive integers where $q$ is a prime power. Then $\mathrm{PN}_{q^n/q}(\epsilon)=\mathrm{Tr}_{q^n/q}(\epsilon^{-1})\cdot \mathrm{N}_{q^n/q}(\epsilon)$ for any $\epsilon\in\Field_{q^{n}}^*$. 
	\end{lem}
	\bp
	For any $\epsilon\in\Field_{q^{n}}^*$, 
	\begin{equation}\nonumber
		\begin{aligned}
			\mathrm{PN}_{q^n/q}(\epsilon)=\sum_{i=0}^{n-1}\Bigg(\underset{j\neq i}{\prod_{j=0}^{n-1}}\epsilon^{q^j}\Bigg)&=\sum_{i=0}^{n-1}\Bigg(\epsilon^{-q^i}\underset{}{\prod_{j=0}^{n-1}}\epsilon^{q^j}\Bigg)\\&=	\Bigg(\sum_{i=0}^{n-1}\epsilon^{-q^i}\Bigg)\cdot\Bigg(\underset{}{\prod_{j=0}^{n-1}}\epsilon^{q^j}\Bigg)\\&=\mathrm{Tr}_{q^n/q}(\epsilon^{-1})\cdot \mathrm{N}_{q^n/q}(\epsilon)
		\end{aligned}
	\end{equation}
	\ep 
Thus, for $a\in\Field_{q}$, to investigate the existence of a primitive normal pair $(\epsilon,f(\epsilon))$, where $f(x) \in \Field_{q^{n}}(x)$ with $\mathrm{PN}_{\Field_{q^{n}}/\Field_{q}}(\epsilon)=a$, it is sufficient to show that for $a\in\Field_{q}$ and primitive $b\in\Field_{q}^*$, there exists a primitive normal pair $(\epsilon,f(\epsilon))$, where $f(x) \in \Field_{q^{n}}(x)$ and $\mathrm{Tr}_{q^n/q}(\epsilon^{-1})=ab^{-1}$ and $\mathrm{N}_{q^n/q}(\epsilon)=b$. Indeed, for any $f(x)\in\Field_{q^{n}}(x)$, the existence of primitive normal pairs $(\epsilon,f(\epsilon))$ together with prescribed trace or norm (or both) has been an interesting domain of research and numerous researchers has worked in the direction \cite{AMS,SA, WS, SH, CW,KHAS}. In this article, we identify those pairs $(q,n)$ such that for $f(x)\in\Field_{q^{n}}(x)$, the field $\Field_{q^{n}}$ contains a primitive normal pair $(\epsilon,f(\epsilon))$ over $\Field_{q}$ such that ${\mathrm{PN}}_{q^n/q}(\epsilon)=a$ for any $a\in\Field_{q}$. 
	
	We will define the following sets, which will have significant importance in this article, for $m_{1},m_{2}\in\mathbb{N}\cup \{0\}$. 
	\begin{enumerate}
		\item[1.]  Let us define $\mathcal{Q}_{q,n}(m_{1},m_{2})$ to be the set contains with the rational functions $f(x)=\frac{f_{1}(x)}{f_{2}(x)}$, where $f_{1}$ and $f_{2}$ co prime irreducible polynomials over $\Field_{q^n}$ such that $x\nmid f_{1},f_{2}$ with $deg(f_{1})=m_{1}$ and $deg(f_{2})=m_{2}$.
		\item[2.] Let $\mathcal{S}_{m_{1},m_{2}}$ appears to be the set containing the pairs $(q,n)\in\mathbb{N}\times\mathbb{N}$ such that for any $f\in\mathcal{Q}_{q,n}(m_{1},m_{2})$, $a\in\Field_{q}$, and any primitive $b\in\Field_q^*$, there exists a primitive normal pair $(\epsilon,f(\epsilon))\in\Field_{q^{n}}^*\times\Field_{q^{n}}^*$ for which $\mathrm{Tr}_{q^n/q}(\epsilon^{-1})=ab^{-1}$ and $\mathrm{N}_{q^n/q}(\epsilon)=b$.\\
		\item[3.] Define , $\mathcal{Q}_{q,n}(m)=\bigcup_{m_{1}+m_{2}=m}\mathcal{Q}_{q,n}(m_{1},m_{2})$ and $\mathcal{S}_{m}=\bigcap_{m_{1}+m_{2}=m}\mathcal{S}_{m_{1},m_{2}}$.
	\end{enumerate}
	
Clearly, $(q,1)\notin \mathcal{S}_{m_{1},m_{2}}$ as in that case we get that $\mathrm{PN}_{q^n/q}(\epsilon)=\epsilon$. Hence $(q,1)$ to be in $\mathcal{S}_{m_{1},m_{2}}$, every pair $(\epsilon,f(\epsilon))$ in $\Field_{q}$ must be primitive normal,  for any $f\in\mathcal{Q}_{q,n}(m_{1},m_{2})$, which is possible only if $q-1$ is prime, that is, if $p=2$. We assume that $f(x)=x+1$. Then it implies that $(1,0)$ is a primitive normal pair. Due to complexity, we have not discussed the cases $n=2,3,4$ in this article, while in future we shall try the remaining cases further.

The structure of this article is as follows. Fundamental notations and definitions that will be used all through this article are mentioned in Section $\ref{S2}$. In Section $\ref{S3}$, we prove a sufficient condition for achieving our main objective. Subsequently, in Section $\ref{S4}$, we introduce the prime sieve condition, which relaxes the sufficient condition. Lastly, we demonstrate the application of the results in the preceding sections  by considering finite fields with characteristic $7$ and $m=3$. Specifically, we derive a subset of $\mathcal{S}_{3}$.
	\setcounter{section}{1}
	\section{Preliminaries}\lb{S2}
	This section presents  a preliminary outline of essential concepts, symbols, and definitions that will be used throughout this article. In this context, $n$ signifies a positive integer, $q$ denotes any prime power, and $\mathbb{F}_{q}$ indicates the finite field containing $q$ elements.
	\begin{defn} (Character).  Let $A$ be an abelain group and $U$ be the subset of complex numbers containing elements on the circle with unit modulus. A character $\chi$ of $A$ is a homomorphism from $A$ into $U$, i.e., $\chi(a_{1}a_{2})=\chi(a_{1})\chi(a_{2})$ for all $a_{1},a_{2}\in A$.
	\end{defn}
	The character $\chi_{1}$ defined by $\chi_{1}(a)=1$ for all $a\in A$, is said to be the trivial character of $A$. Moreover, the collection of all characters of $A$, denoted as $\widehat{A}$, forms a group under multiplication and $A\cong \widehat{A}$. Further, since $\Field_{q^{n}}^*\cong \widehat{\Field_{q^{n}}^*}$, then for any $d|q^n-1$, there are $\phi(d)$ multiplicative characters of order $d$.
	
	In the context of a finite field $\Field_{q^{n}}$, a multiplicative character relates to the multiplicative group $\Field_{q^{n}}^*$, whereas an additive character corresponds to the additive group $\Field_{q^{n}}$. Any multiplicative character can $\chi$, associated to $\Field_{q^{n}}^*$, can be extended to $\Field_{q^{n}}$ by the following rule
	$$
	\chi(0)=\begin{cases}
		
		1, ~\text{if} ~\chi= \chi_{1},  \\
		0, ~\text{if} ~\chi\neq \chi_{1}.
	\end{cases}\\$$
	\begin{defn}
		($e$-free element). Let $\epsilon \in \Field_{q^{n}}^{*}$ and $e$ be any divisor of $q^n-1$. Then, $\epsilon$ is said to be an $e$-free, if $\epsilon=\delta^{d}$, where $\delta \in\Field_{q^{n}}^*$ and $d|e$ implies that $d=1$. Clearly, $\epsilon\in\Field_{q^{n}}^*$ is primitive if and only if $\epsilon$ is $(q^{n}-1)$-free.
	\end{defn}
	For any $e|(q^n-1)$, the characteristic function determining the subset of $e$-free elements of $\Field_{q^n}^*$ is given by 
	\begin{equation}\lb{eq1}
		\begin{aligned}
			\rho_{e}:\Field_{q^n}^*\rightarrow \{0,1\};\epsilon \mapsto{}\theta(e)\sum_{d|e}\frac{\mu(d)}{\phi(d)}\sum_{\chi_{d}}\chi_{d}(\epsilon),
		\end{aligned}
	\end{equation}
	where $\theta(e):=\frac{\phi(e)}{e}$, $\chi_{d}$ represents a multiplicative character of order $d$ in $\widehat{\Field_{q^n}^*}$ and $\mu$ is the M\"{o}bius function.

	The additive group $\Field_{q^{n}}$ becomes an $\Field_{q}[x]$-module according to the following rule.
	$$f\circ\epsilon=\sum_{i=0}^{r}{a_{i}\epsilon^{q^{i}}}; ~\text{for}~ \epsilon\in\Field_{q^{n}}~ \text{and}~ g(x)=\sum_{i=0}^{r}{a_{i}x^{i}}\in \Field_{q}[x].$$ For $\epsilon\in\Field_{q^{n}}$, the $\Field_{q}$-order of $\epsilon$ is the monic $\Field_{q}$-divisor $f$ of $x^n-1$ of minimal degree such
	that $f\circ \epsilon= 0$.
	\begin{defn}($g$-free element).
	Let $g|x^n-1$ and $\epsilon\in\Field_{q^{n}}$. Suppose that for any $h|g$ and $\delta\in\Field_{q^{n}}$, $\epsilon=h\circ\delta$ implies $h=1$. Then $\epsilon\in\Field_{q^{n}}$ is said to be $g$-free. It is straightforward to observe that, any element $\epsilon\in\Field_{q^{n}}$ is normal if and only if $\epsilon$ is $(x^n-1)$-free.
	\end{defn}
	For any $g|x^n-1$, the characteristic function determining the subset of $g$-free elements in $\Field_{q^{n}}$ is given by
	\begin{equation}\lb{eq2}
		\begin{aligned}
			\kappa_{g}:\Field_{q^n}\mapsto \{0,1\};\epsilon\xrightarrow{}\Theta(g)\sum_{h|g}\frac{\mu_{q}(h)}{\Phi_{q}(h)}\sum_{\la_{h}}\la_{h}(\epsilon),
		\end{aligned}
	\end{equation}
	where $\Theta(g):=\frac{\Phi_{q}(g)}{q^{deg(g)}}$, $\la_{h}$ stands for any additive character of $\Field_{q}$-order $h$ in $\widehat{\Field_{q^n}}$ and $\mu_{q}$ is the M\"{o}bius function for the set of polynomials over $\Field_{q}$ is defined as follows:
	
	$$
	\mu_{q}(f)=\begin{cases}
		
		(-1)^{r}, ~\text{if $f$ is product of $r$ distinct monic irreducible polynomials},  \\
		0,\;\;\;\;\;\;\;\;\;\text{otherwise}.
	\end{cases}\\$$

For any $a\in\Field_{q}$, the characteristic function for the subset of $\Field_{q^n}$ containing elements with $\mathrm{Tr}_{{q^n}/q}(\epsilon)=a$ is given as follows:
	\begin{equation}\nonumber
		\tau_{a}:\Field_{q^n}\rightarrow \{0,1\};\epsilon \mapsto \frac{1}{q}\sum_{\la\in{\widehat{{\Field}_{q}}}}^{}\la(\mathrm{Tr}_{{q^n}/q}(\epsilon)-a).
	\end{equation}
 Now, any additive character $\la$ of $\Field_{q}$ can be represented using the canonical additive character $\la_{0}$ as $\la(\epsilon)=\la_{0}(t\epsilon)$, where $t$ is an element of $\Field_{q}$ that corresponds to $\la$. Thus
	\begin{equation}\lb{eq3}
		\begin{aligned}
			\tau_{a}(\epsilon)&= \frac{1}{q}\sum_{t\in\Field_{q}}^{}\la_{0}(\mathrm{Tr}_{{q^n}/q}(t\epsilon)-ta)\\
			&=\frac{1}{q}\sum_{t\in\Field_{q}}^{}{\widehat{\la}_{0}}(t\epsilon)\la_{0}(-ta),
		\end{aligned}
	\end{equation}
where $\widehat{\la}_{0}$ is the additive character of $\Field_{q^n}$, that is given by $\widehat{\la}_{0}(\epsilon)=\la_{0}(\mathrm{Tr}_{{q^n}/q}(\epsilon))$.
	
	Moreover, for $c\in\Field_q^*$, the characteristic function for the subset of $\Field_{q^n}^*$ containing elements with $\mathrm{N}_{{q^n}/q}(\epsilon)=c$ is defined as follows:
	\begin{equation}\nonumber
		\begin{aligned}
			\eta_{c}:\Field_{q^n}^*\rightarrow \{0,1\};\epsilon \mapsto \frac{1}{q-1}\sum_{\chi\in{\widehat{{\Field}_{q}^*}}}^{}\chi(\mathrm{N}_{{q^n}/q}(\epsilon)c^{-1}).
		\end{aligned}
	\end{equation} 
	Let $\chi_{q-1}$ be a multiplicative character of order $q-1$. Thus any $\chi$ of $\widehat{{\Field}_{q}^*}$ can be expressed in terms of $\chi_{q-1}$ as $\chi(\epsilon)=\chi_{q-1}(\epsilon^{i})$ for some positive integer $i\in\{1,2,\ldots,q-1\}$. Thus
	\begin{equation}
		\begin{aligned}\lb{eq4}
			\eta_{c}(\epsilon)&=\frac{1}{q-1}\sum_{i=1}^{q-1}\chi_{q-1}^{i}(\mathrm{N}_{{q^n}/q}(\epsilon)c^{-1})\\
			&=\frac{1}{q-1}\sum_{i=1}^{q-1}\tilde{\chi}^{i}(\epsilon)\chi_{q-1}(a^{-i})
		\end{aligned}
	\end{equation} 
	where $\tilde{\chi}=\chi_{q-1}\circ \mathrm{N}_{{q^n}/q}$ is a multiplicative character of $\Field_{q^n}^*$. Following \cite{AMS}, the order of $\tilde{\chi}$ is $q-1$ and there exists a multiplicative character $\chi_{q^{n}-1}$ of order $q^{n}-1$ such that $\tilde{\chi}=\chi_{q^n-1}^{{q^{n}-1}/{q-1}}$.

	The following lemmas has proved by Wan and Fu {\cite{LDQ}} and are crucial for proving our main result as well as the modified prime sieve.
	\begin{lem}\textbf{({\cite{LDQ}}, Theorem 5.5)}\lb{L1}
	Consider $f(x)=\prod_{i=1}^{r}f_{i}(x)^{a_{i}}\in\mathbb{F}_{q^n}(x)$ be such that $f_{i}$'s are irreducible polynomials over $\Field_{q^{n}}$ and $a_{i}$'s are nonzero integers. Suppose that $\chi\in\widehat{{\Field}_{q}^*}$ be a multiplicative character having order $d$. Further, assume $f(x)$ to be a rational function, which is not equal to $h(x)^{d}$, for $h(x)\in\mathbb{F}_{q^n}(x)$, the of rational functions. Then  $$\Bigg|\sum_{\epsilon\in\Field_{q^n}, f(\epsilon)\neq 0,\infty }^{}\chi(f(\epsilon))\Bigg|\leq\Bigg(\sum_{i=1}^{r}deg(f_{i})-1\Bigg)q^{n/2}.$$		
	\end{lem}
	\begin{lem}\textbf{({\cite{LDQ}}, Theorem 5.6)}\lb{L2}
	Let $f(x)$, $g(x)$ $\in\mathbb{F}_{q^n}(x)$ be rational functions over $\mathbb{F}_{q^n}$. Express $f(x)$ as $\prod_{i=1}^{r} f_{i}(x)^{a_i}$, where each $f_i(x)$ is an irreducible polynomial over the field $\mathbb{F}_{q^n}$, and $a_i$'s are nonzero integers $(i=1,2,\ldots,r)$. Let $D_{1}=\sum_{i=1}^{r}deg(f_{i})$, $D_{2}=max(deg(g(x)),0)$, $D_{3}$ represents the degree of the denominator of $g(x)$ and $D_{4}$ denotes the sum of the degrees of irreducible polynomials dividing the denominator of $g(x)$ (excluding those equal to $f_{i}(x)$, for $i=1,2,\ldots,r$). Consider $\chi$ as a multiplicative character of $\Field_{q^n}^{*}$ and $\lambda$ as a nontrivial additive character of $\Field_{q^{n}}$. Further, assume that $g(x)\neq h(x)^{q^n}-h(x)$, where $h(x)\in\Field_{q^n}(x)$. Then we have
	\begin{equation} \nonumber
		\begin{aligned}
			\Bigg|\sum_{\epsilon \in \Field_{q^{n}}, f(\epsilon)\neq 0,\infty, g(\epsilon)\neq\infty }\chi(f(\epsilon))\la(g(\epsilon))\Bigg|\leq(D_{1}+D_{2}+D_{3}+D_{4}-1)q^{n/2}.
		\end{aligned}
	\end{equation}
	\end{lem}
For $l$, a positive integer (or a monic polynomial over $\Field_{q}$), we use $\om(l)$ to represent  the cardinality of the set which contains distinct prime divisors (irreducible factors ) of $l$ and $W(l)$ to represent the cardinality of the set which contains square-free divisors (square-free factors) of $l$, that is $W(l)=2^{\om(l)}$.
	\begin{lem}\textbf{({\cite{SS}}, Lemma 3.7)}\lb{L2.3}
		Let $r>0$ be a real number and $m$ be a positive integer. Then $W(m)<\mathcal{C}\cdot m^{\frac{1}{r}}$, where $\mathcal{C}=\frac{2^w}{{({p_{1}p_{2}\ldots p_{w})}}^{\frac{1}{r}}}$ and $p_{1},p_{2},\ldots,p_{w}$ are primes $\leq 2^{r}$ that divide $m$.
	\end{lem}
	\begin{lem}\textbf{({\cite{HWRJ}}, Lemma 2.9)}\lb{L2.4}
	Suppose that $q$ be a prime power, $n$ be a natural number and $n'$=gcd$(n,q-1)$. Then we have $W(x^n-1)\leq 2^{\frac{1}{2}\{n+n'\}}$, which gives $W(x^n-1)\leq 2^n$. Further, $W(x^n-1)=2^n$ if and only if $n|q-1$. In addition, if  $n\nmid q-1$, then $W(x^n-1)\leq 2^{\frac{3}{4}n}$.
	\end{lem}
	We know that norm of a primitive element is also primitive. Moreover, Sharma et al. {{\cite{AMS}}} has proved the following lemma in a more general context.
	\begin{lem}\textbf{({\cite{AMS}}, Lemma 3.1)}\lb{L5}
		Let $e$ be a positive divisor of $q^n-1$ and $\delta=$gcd($e,q-1$). Additionally, let $Q_{e}$ represents the largest divisor of $e$ for which gcd($Q_{e},\delta$)=1. Then an element $\epsilon\in\Field_{q^n}^*$ is $e$-free $\iff$ $\epsilon$ is $Q_{e}$-free and $N_{{q^n}/q}(\epsilon)$ is $\delta$-free. 
	\end{lem}
	\section{Main Result}\lb{S3}
	Let $e_{1},e_{2}|q^n-1$ and $g_{1},g_{2}|x^n-1$. Let $\delta=$gcd($e_{1},q-1$) and $Q_{e_{1}}$ be the largest divisor of $e_{1}$ such that gcd($Q_{e_{1}},\delta$)$=1$. Let $m_{1},m_{2}\in\mathbb{N}\cup \{0\}$ be such that $1\leq m_{1}+m_{2}<q^n$. Also, let $f(x)\in\mathcal{Q}_{q,n}(m)$, $a\in\Field_{q}$ and $b\in\Field_{q}^*$ be $\delta$-free element. We denote $\mathfrak{M}_{f,a,b}(Q_{e_{1}},e_{2},g_{1},g_{2})$ as the number of $\epsilon\in\Field_{q^{n}}^*$ such that $\epsilon$ is $Q_{e_{1}}$-free, $g_{1}$-free and $f(\epsilon)$ is $e_{2}$-free, $g_{2}$-free with $\mathrm{Tr}_{q^{n}/q}(\epsilon^{-1})=ab^{-1}$ and $\mathrm{N}_{q^{n}/q}(\epsilon)=b$. Let us abbreviate, $Q:=Q_{q^n-1}$. 
	
	We hereby prove the following inequality like sufficient condition.
	\begin{thm}\lb{T3.1}
		Let $n,m,q\in\mathbb{N}$ be such that $q$ is a prime power and $n\geq 5$. Assume that
		$$q^{\frac{n}{2}-2}>
		(2m+2)W(Q){W(q^n-1)}{W(x^n-1)}^2.$$
		Then $(q,n)\in\mathcal{S}_{m}$.
	\end{thm}
	\bp
	Suppose that $\mathcal{U}$ be the set containing zeros and poles of $f$ and $\mathcal{U}_{1}=\mathcal{U}\cup \{0\}$. Then by the definition $\mathfrak{M}_{f,a,b}(Q_{e_{1}},e_{2},g_{1},g_{2})$ is given by
	\begin{equation}\nonumber
		\begin{aligned}
			\sum_{\epsilon\in\Field_{q^{n}}\smallsetminus \mathcal{U}_{1}}^{}\rho_{Q_{e_{1}}}(\epsilon)\rho_{e_{2}}(f(\epsilon))\kappa_{g_{1}}(\epsilon)\kappa_{g_{2}}(f(\epsilon))\tau_{ab^{-1}}(\epsilon^{-1})\eta_{b}(\epsilon).
		\end{aligned}
	\end{equation}
	Using (\ref{eq1}), (\ref{eq2}), (\ref{eq3}) and (\ref{eq4}) we have
	\begin{equation}\lb{E5}
		\begin{aligned}
			\mathfrak{M}_{f,a,b}(Q_{e_{1}},e_{2},g_{1},g_{2})=\mathcal{H}\underset{\underset{h_{1}|g_{1},h_{2}|g_{2}}{d_{1}|Q_{e_{1}}, d_{2}|e_{2}}}{\sum}\frac{\mu}{\phi}(d_{1},d_{2},h_{1},h_{2})\underset{\underset{\la_{h_{1},\la_{h_{2}}}}{\chi_{d_{1}},\chi_{d_{2}}}}{\sum}\boldsymbol{\chi}_{f,a,b}(d_{1},d_{2},h_{1},h_{2}) 
		\end{aligned}
	\end{equation}
	where $\mathcal{H}=\frac{\theta(Q_{e_{1}})\theta(e_{2})\Theta(g_{1})\Theta(g_{2})}{q(q-1)}$, $\frac{\mu}{\phi}(d_{1},d_{2},h_{1},h_{2})=\frac{\mu(d_{1})\mu(d_{2})\mu_{q}(h_{1})\mu_{q}(h_{2})}{\phi(d_{1})\phi(d_{2})\Phi_{q}(h_{1})\Phi_{q}(h_{2})}$ and
	\begin{equation}\nonumber
		\begin{aligned}
			\boldsymbol{\chi}_{f,a,b}(d_{1},d_{2},h_{1},h_{2}) &=\sum_{i=1}^{q-1}\sum_{t\in\Field_{q}}^{}\chi_{q-1}(b^{-i})\la_{0}(-ab^{-1}t)\sum_{\epsilon\in\Field_{q^{n}}\smallsetminus \mathcal{U}_{1}}^{}\chi_{d_{1}}(\epsilon)\chi_{d_{2}}(f(\epsilon))\la_{h_{1}}(\epsilon)\\&\times\la_{h_{2}}(f(\epsilon))\tilde{\chi}^{i}(\epsilon)\widehat{\la_{0}}(t\epsilon^{-1}).
		\end{aligned}
	\end{equation}
	Since $\chi_{q^n-1}$ is a multiplicative character of order $q^n-1$ in the cyclic group $\widehat{\Field_{q^{n}}^*}$, there exist $c_{i}\in\{0,1,2,\ldots,q^n-2\}$ such that $\chi_{d_{i}}(\epsilon)=\chi_{q^n-1}(\epsilon^{c_{i}})$ for $i=1,2$. 
	Furthermore, there exist $y_{1},y_{2}\in\Field_{q^{n}}$ such that $\la_{h_{i}}(\epsilon)=\widehat{\la_{0}}(y_{i}\epsilon)$, for $i=1,2$. Thus
	\begin{equation}\nonumber
		\begin{aligned}
			\boldsymbol{\chi}_{f,a,b}(d_{1},d_{2},h_{1},h_{2}) =&\sum_{i=1}^{q-1}\sum_{t\in\Field_{q}}^{}\chi_{q-1}(b^{-i})\la_{0}(-ab^{-1}t)\sum_{\epsilon\in\Field_{q^{n}}\smallsetminus \mathcal{U}_{1}}^{}\chi_{q^n-1}(\epsilon^{c_{1}+\frac{q^n-1}{q-1}i}f(\epsilon)^{c_{2}})\\&\times\widehat{\la_{0}}(y_{1}\epsilon+t\epsilon^{-1}+y_{2}f(\epsilon))\\
			=&\sum_{i=1}^{q-1}\sum_{t\in\Field_{q}}^{}\chi_{q-1}(b^{-i})\la_{0}(-ab^{-1}t)\sum_{\epsilon\in\Field_{q^{n}}\smallsetminus \mathcal{U}_{1}}^{}\chi_{q^n-1}(F(\epsilon))\widehat{\la_{0}}(G(\epsilon)),
		\end{aligned}
	\end{equation}
	where $F(x)=x^{c_{1}+\frac{q^n-1}{q-1}i}{f(x)}^{c_{2}}\in\Field_{q^{n}}(x)$ and $G(x)=y_{1}x+tx^{-1}+y_{2}f(x)\in\Field_{q^{n}}(x)$.   Firstly, let us consider the situation $G(x)\neq {\mathcal{L}(x)}^{q^n}-{\mathcal{L}(x)}$ for any $\mathcal{L}(x)\in\Field_{q^{n}}(x)$. Here we arrive at the following possibilities.\\
{\textbf{Case 1:}}
	If $m_{2}+1\geq m_{1}$, then as mentioned in Lemma \ref{L2} we have $D_{2}=1$, and
	\begin{equation}\nonumber
		\begin{aligned}
			|\boldsymbol{\chi}_{f,a}(d_{1},d_{2},h_{1},h_{2})|\leq (2m+2)(q-1)q^{\frac{n}{2}+1}.\end{aligned}
	\end{equation}
{\textbf{Case 2:}} If $m_{2}+1<m_{1}$, then we have $D_{2}=m_{1}-m_{2}$ and
	\begin{equation}\nonumber
		\begin{aligned}
			|\boldsymbol{\chi}_{f,a}(d_{1},d_{2},h_{1},h_{2}) |\leq (2m+1)(q-1)q^{\frac{n}{2}+1}.
		\end{aligned}
	\end{equation} Next, we assume that $G(x)={\mathcal{L}(x)}^{q^n}-{\mathcal{L}(x)}$ for some $\mathcal{L}(x)\in\Field_{q^{n}}(x)$.
	Then we have
	\begin{equation}\lb{E6}
		\begin{aligned}
			y_{1}x+tx^{-1}+y_{2}f(x)={\mathcal{L}(x)}^{q^n}-{\mathcal{L}(x)}.
		\end{aligned}
	\end{equation}
	We claim that the above equation is feasible only if $y_{1}=y_{2}=t=0$. Let us write $\mathcal{L}(x)=\frac{l_{1}(x)}{l_{2}(x)}$ with gcd$(l_{1},l_{2})=1$, which gives that
	\begin{equation}\lb{E7}
		\begin{aligned}
			xf_{2}(x)({l_{1}(x)}^{q^n}-l_{1}(x){l_{2}(x)}^{q^n-1})={l_{2}(x)}^{q^n}(y_{1}x^2f_{2}(x)+tf_{2}(x)+y_{2}xf_{1}(x)).
		\end{aligned}
	\end{equation}
	Since $({l_{1}(x)}^{q^n}-l_{1}(x){l_{2}(x)}^{q^n-1}, {l_{2}(x)}^{q^n})=1$, it implies that ${l_{2}(x)}^{q^n}|xf_{2}(x)$. Further, since $f\in\mathcal{Q}_{q,n}(m)$, we have $f_{2}(x)|{l_{2}(x)}^{q^n}\implies {l_{2}(x)}^{q^n}=k{f_{2}}(x)$ for some $k\in\Field_{q^{n}}[x]$, which further implies ${l_{2}(x)}^{q^n}=f_{2}(x)$ or ${l_{2}(x)}^{q^n}=xf_{2}(x)$. The earlier is possible only if $l_{2}(x)=w$, where $w\in\Field_{q^{n}}^*$. Then (\ref{E7}) becomes
	\begin{equation}\nonumber
		x({l_{1}(x)}^{q^n}-l_{1}(x))=(y_{1}x^2w+tw+y_{2}xf_{1}(x)),
	\end{equation}
	and this forces that $t=0$. Substituting it to the above yields, ${l_{1}(x)}^{q^n}-l_{1}(x)=y_{1}xw+y_{2}f_{1}(x)$, which happens only if $l_{1}$ is nonzero constant and $y_{1}=y_{2}=0$. Now, let us consider the latter possibility, that is, ${l_{2}(x)}^{q^n}=xf_{2}(x)$. This gives $x|l_{2}(x)$, which further gives $x|f_{2}(x)$, a contradiction. Hence we have $t=y_{1}=y_{2}=0$, that is, $h_{1}=h_{2}=1, t=0$. In addition to this, let us consider the following possibilities.\\
{\textbf{Case 1:}} If $F(x)\neq {\mathcal{R}(x)}^{q^n-1}$ for any $\mathcal{R}(x)\in\Field_{q^{n}}(x)$, then it follows from Lemma \ref{L1}, that
	\begin{equation}\nonumber
		\begin{aligned}
			|\boldsymbol{\chi}_{f,a}(d_{1},d_{2},h_{1},h_{2}) |\leq m(q-1)q^{\frac{n}{2}+1}.
		\end{aligned}
	\end{equation}
{\textbf{Case 2:}} Here, we consider the case when $F(x)=\mathcal{R}(x)^{q^n-1}$ for some $\mathcal{R}(x)\in\Field_{q^{n}}(x)$, where $\mathcal{R}(x)=\frac{r_{1}(x)}{r_{2}(x)}$ with gcd$(r_{1},r_{2})=1$. Then we have $x^{c_{1}+\frac{q^n-1}{q-1}i}\bigg({\frac{f_{1}(x)}{f_{2}(x)}}\bigg)^{c_{2}}=\bigg({\frac{r_{1}(x)}{r_{2}(x)}}\bigg)^{q^{n}-1}$, that is,
	\begin{equation}\lb{E8}
		\begin{aligned}
			x^{c_{1}+\frac{q^n-1}{q-1}i}{f_{1}(x)}^{c_{2}}{r_{2}(x)}^{q^{n}-1}={f_{2}(x)}^{c_{2}}{r_{1}(x)}^{q^{n}-1}.
		\end{aligned}
	\end{equation}
	We now show that equation (\ref{E8}) is feasible only if $c_{1}=c_{2}=0$. For this, first we show that if $c_{2}$ is $0$, then $c_{1}$ must be $0$. Suppose that $c_{2}=0$. From equation (\ref{E8}), it follows that $c_{1}+\frac{q^n-1}{q-1}i=k_{1}(q^n-1)$ for some positive integer $k_{1}$.  Following \cite{AMS}, it happens only if $c_{1}=0$. Next if possible, let $c_{2}\neq 0$. Again, $c_{1}+\frac{q^n-1}{q-1}i>0$ forces that either $x|f_{2}(x)$ or $x|r_{1}(x)$. We may assume that $x|r_{1}(x)$, as $x\nmid f_{2}(x)$. Rewrite equation (\ref{E8}) as
	\begin{equation}\nonumber
		\begin{aligned}
			{f_{1}(x)}^{c_{2}}{r_{2}(x)}^{q^{n}-1}={r_{1}'(x)}^{q^{n}-1}{f_{2}(x)}^{c_{2}}x^{q^n-1-\frac{q^n-1}{q-1}i-c_{1}},
		\end{aligned}
	\end{equation}
	where $r_{1}'(x)=\frac{r_{1}(x)}{x}$. Let us discuss the following possible cases.\\
	\textbf{Case 2.1.}  $q^n-1-\frac{q^n-1}{q-1}i-c_{1}>0$. Since gcd$(r_{1}(x),r_{2}(x))=1$, we must have $x|f_{1}(x)$, a contradiction.\\
	\textbf{Case 2.2.}  $q^n-1-\frac{q^n-1}{q-1}i-c_{1}=0$. In this case, we have $	{f_{1}(x)}^{c_{2}}{r_{2}(x)}^{q^{n}-1}={r_{1}'(x)}^{q^{n}-1}{f_{2}(x)}^{c_{2}}$. Since gcd$(f_{1}(x),f_{2}(x))=1$, the latter gives $f_{2}(x)|r_{2}(x)$, which further implies that ${f_{1}(x)}^{c_{2}}{r_{2}'(x)}^{q^{n}-1}{f_{2}(x)}^{q^n-1-c_{2}}={r_{1}'(x)}^{q^{n}-1}$, where $r_{2}'(x)=\frac{r_{2}(x)}{f_{2}(x)}$. Since $q^n-1-c_{2}>0$, we must have $f_{2}(x)|r_{1}'(x)$, a contradiction.\\
	\textbf{Case 2.3.}  $q^n-1-\frac{q^n-1}{q-1}i-c_{1}<0$. As $x\nmid f_{2}(x)$, so we have $x|r_{1}'(x )$, which gives $	{f_{1}(x)}^{c_{2}}{r_{2}(x)}^{q^{n}-1}={r_{1}''(x)}^{q^{n}-1}{f_{2}(x)}^{c_{2}}x^{2(q^n-1)-\frac{q^n-1}{q-1}i-c_{1}}$, where $r_{1}''(x)=\frac{r_{1}'(x)}{x}$. Here $2(q^n-1)-\frac{q^n-1}{q-1}i-c_{1}>0$ implies that $x|f_{1}(x)$, a contradiction. Thus, it follows that  $F(x)\neq {\mathcal{R}(x)}^{q^n-1}$ for any $\mathcal{R}(x)\in\Field_{q^{n}}(x)$. Hence, $c_{1}=c_{2}=0$, that is, $d_{1}=d_{2}=1$.
	
	Thus, if $(d_{1},d_{2},h_{1},h_{2})\neq (1,1,1,1)$, then based one the above discussions, we get that
	\begin{equation}\nonumber
		|\boldsymbol{\chi}_{f,a,b}(d_{1},d_{2},h_{1},h_{2}) |\leq (2m+2)(q-1)q^{\frac{n}{2}+1}.\\
	\end{equation}
	Further, we have
	\begin{equation}\nonumber
		\begin{aligned}
			\boldsymbol{\chi}_{f,a,b}(1,1,1,1)&=\sum_{i=1}^{q-1}\sum_{t\in\Field_{q}}^{}\chi_{q-1}(b^{-i})\la_{0}(-ab^{-1}t)\sum_{\epsilon\in\Field_{q^{n}}\smallsetminus \mathcal{U}_{1}}^{}{\chi}_{q^n-1}(\epsilon^{\frac{q^n-1}{q-1}i})\widehat{\la_{0}}(t\epsilon^{-1})\\
			&=(q^n-|\mathcal{U}_{1}|)+V_{1}+V_{2},
		\end{aligned}
	\end{equation} where
	\begin{equation}\nonumber
		\begin{aligned}
			V_{1}=\sum_{i=1}^{q-2}\chi_{q-1}(b^{-i})\sum_{\epsilon\in\Field_{q^{n}}\smallsetminus \mathcal{U}_{1}}^{}{\chi}_{q^n-1}(\epsilon^{\frac{q^n-1}{q-1}i})
		\end{aligned}
	\end{equation}
	$$\text{and}$$
	\begin{equation}\nonumber
		\begin{aligned}
			V_{2}=\sum_{i=1}^{q-1}\sum_{t\in\Field_{q}^*}^{}\chi_{q-1}(b^{-i})\la_{0}(-ab^{-1}t)\sum_{\epsilon\in\Field_{q^{n}}\smallsetminus \mathcal{U}_{1}}^{}{\chi}_{q^n-1}(\epsilon^{\frac{q^n-1}{q-1}i})\widehat{\la_{0}}(t\epsilon^{-1}).
		\end{aligned}
	\end{equation} 
	Now, let us find upper bounds of $|V_{1}|$ and $|V_{2}|$. Note that for $i\in\{1,2,\ldots, q-2\}$, ${\chi}_{q^n-1}^{\frac{q^n-1}{q-1}i}$ is a nontrivial character and thus $\sum_{\epsilon\in\Field_{q^{n}}^*}^{}{\chi}_{q^n-1}^{\frac{q^n-1}{q-1}i}(\epsilon)=0$. Hence, we get $|V_{1}|\leq m(q-2)$. Moreover, for any $t\in\Field_{q}^*$, $tx^{-1}$ is not of the form $H(x)^{q^n}-H(x)$, for any $H(x)\in\Field_{q^{n}}(x)$. Then, we have $|V_{2}|\leq (q^{n/2}+m)(q-1)^2$.\\
	
		Therefore from (\ref{E5}), we get
	\begin{equation}\lb{E9}
		\begin{aligned}
			\mathfrak{M}_{f,a,b}(Q_{e_{1}},e_{2},g_{1},g_{2})&\geq \mathcal{H}\{q^n-(m+1)-m(q-2)-(q-1)^2(q^{n/2}+m)\\&- (2m+2)q^{\frac{n}{2}+2}(W(Q_{e_{1}})W(e_{2})W(g_{1})W(g_{2})-1)\}\\
			&\geq \mathcal{H}\{(q^n- (2m+2)q^{\frac{n}{2}+2}W(Q_{e_{1}})W(e_{2})W(g_{1})W(g_{2})\}.
		\end{aligned}
	\end{equation}
	Hence $\mathfrak{M}_{f,a.b}(Q_{e_{1}},e_{2},g_{1},g_{2})>0$, if we have $q^{\frac{n}{2}-2}> (2m+2)W(Q_{e_{1}})W(e_{2})W(g_{1})W(g_{2})$. Consequently, we have $(q,n)\in\mathcal{S}_{m}$ by choosing $e_{1}=e_{2}=q^n-1$ and $g_{1}=g_{2}=x^n-1$, that is provided
	$$q^{\frac{n}{2}-2}>
	(2m+2)W(Q){W(q^n-1)}{W(x^n-1)}^2.$$
	\ep
	\section{Prime Sieve}\lb{S4}
	
	In this section, we provide the prime sieve inequality initially instigated by Kapetanakis in \cite{GK}, and subsequently employ it following certain modifications.
	\begin{lem}\lb{L4.1}
		Let $e'|Q$ and $p_{1}',p_{2}',\ldots,p_{u}'$ be the remaining distinct prime divisors of $Q$, let $e| q^n-1$ and $p_{1},p_{2},\ldots,p_{r}$ be the remaining distinct primes dividing $q^n-1$. Moreover, let $g$ be a divisor of $x^n-1$ and $g_{1},g_{2},\ldots,g_{s}$ be the remaining distinct irreducible factors of $x^n-1$. Abbreviate $\mathfrak{M}_{f,a,b}(Q,q^n-1,x^n-1,x^n-1)$ to $\mathfrak{M}_{f,a,b}$. Then 
		\begin{equation}
			\begin{aligned}
				\mathfrak{M}_{f,a,b}\geq \sum_{i=1}^{u}\mathfrak{M}_{f,a.b}(p_{i}&'e',e,g,g)+\sum_{i=1}^{r}	\mathfrak{M}_{f,a,b}(e',p_{i}e,g,g)+\sum_{j=1}^{s}\mathfrak{M}_{f,a,b}(e',e,gg_{j},g)\\
				&+\sum_{j=1}^{s}\mathfrak{M}_{f,a,b}(e',e,g,gg_{j})-(u+r+2s-1)\mathfrak{M}_{f,a,b}(e',e,g,g).\\
			\end{aligned}
		\end{equation}
	\end{lem} 
	Upper bounds of certain differences are given in the following result, which will be needed further.
	\begin{lem}\lb{L4.2}
		Let $e',e,n,q\in\mathbb{N}$, $g\in\Field_{q}[x]$ be such that $q$ is a prime power, $n\geq 5$, $e'|Q$, $e|q^n-1$,  and $g|x^n-1$. Let $P'$ be a prime number which divides $Q$ but not $e'$, let $P$ be a prime number which divides $q^n-1$ but not $e$, and also $I$ be an irreducible polynomial which divides $x^n-1$ but not $g$. Then we get the following bounds:
		\begin{equation}\nonumber
			\begin{aligned}
				|\mathfrak{M}_{f,a,b}(P'e',e,g,g)-\theta(P')&\mathfrak{M}_{f,a,b}(e',e,g,g)|\\&\leq (2m+2)\theta(P'){\theta(e)}\theta(e'){\Theta(g)}^2{W(e)}W(e'){W(g)}^2q^{n/2},\\
			\end{aligned}
		\end{equation} 
		\begin{equation}\nonumber
			\begin{aligned}
				|\mathfrak{M}_{f,a,b}(e',Pe,g,g)-\theta(P)&\mathfrak{M}_{f,a,b}(e',e,g,g)|\\&\leq (2m+2)\theta(P)\theta(e)\theta(e'){\Theta(g)}^2W(e)W(e'){W(g)}^2q^{n/2},\\
			\end{aligned}
		\end{equation}
		\begin{equation}\nonumber
			\begin{aligned}
				|\mathfrak{M}_{f,a,b}(e',e,Ig,g)-\Theta(I)&\mathfrak{M}_{f,a,b}(e',e,g,g)|\\&\leq (2m+2)\Theta(I)\theta(e)\theta(e'){\Theta(g)}^2W(e)W(e'){W(g)}^2q^{n/2},\\
			\end{aligned}
		\end{equation}
		\begin{equation}\nonumber
		\begin{aligned}
			|\mathfrak{M}_{f,a,b}(e',e,g,Ig)-\Theta(I)&\mathfrak{M}_{f,a,b}(e',e,g,g)|\\&\leq (2m+2)\Theta(I)\theta(e)\theta(e'){\Theta(g)}^2W(e)W(e'){W(g)}^2q^{n/2}.\\
		\end{aligned}
		\end{equation}
		\end{lem}
	\bp
	From the definition, we have
	\begin{equation}\nonumber
		\begin{aligned}
		\mathfrak{M}_{f,a,b}(P'e',e,g,g)-&\theta(P')\mathfrak{M}_{f,a}(e',e,g,g)\\&=\mathcal{H}\underset{\underset{h_{1}|g,h_{2}|g}{P'|d_{1}|P'e', d_{2}|e}}{\sum}\frac{\mu}{\phi}(d_{1},d_{2},h_{1},h_{2})\underset{\underset{\la_{h_{1},\la_{h_{2}}}}{\chi_{d_{1}},\chi_{d_{2}}}}{\sum}\boldsymbol{\chi}_{f,a}(d_{1},d_{2},h_{1},h_{2}) . 
		\end{aligned}
	\end{equation}
	By using $|\boldsymbol{\chi}_{f,a}(d_{1},d_{2},h_{1},h_{2})|\leq (2m+2)(q-1)q^{\frac{n}{2}+1}$, we get
	\begin{equation}\nonumber
		\begin{aligned}
			|\mathfrak{M}_{f,a}(&P'e',e,g,g)-\theta(P')\mathfrak{M}_{f,a}(e',e,g,g)|\\&\leq\frac{\theta(P')\theta(e)\theta(e')\Theta(g)^2}{q(q-1)}(2m+2)(q-1)q^{\frac{n}{2}+1}W(e){W(g)}^2(W(P'e')-W(e')). 
		\end{aligned}
	\end{equation}
	Since $W(P'e')=W(P')W(e')=2W(e')$, we have 
	\begin{equation}\nonumber
	\begin{aligned}
		|\mathfrak{M}_{f,a}(P'e',e,g,g)-\theta(P')&\mathfrak{M}_{f,a}(e',e,g,g)|\\&\leq (2m+2)\theta(P'){\theta(e)}\theta(e'){\Theta(g)}^2{W(e)}W(e'){W(g)}^2q^{n/2},\\
	\end{aligned}
\end{equation} 
	The other bounds can also be derived in a similar manner.
	\ep
	\begin{thm}\lb{T4.3}
		Let $e',e,n,q\in\mathbb{N}$, $g\in\Field_{q}[x]$ be such that $q$ is a prime power, $n\geq 5$, $e'|Q$, $e|q^n-1$ and $g|x^n-1$. Let $p_{1}',p_{2}',\ldots,p_{u}'$ be the distinct primes dividing $Q$ but not $e'$, let $p_{1},p_{2},\ldots,p_{r}$ be the distinct primes dividing $e$ but not $q^n-1$, and $g_{1},g_{2},\ldots,g_{s}$ be the distinct irreducible factor of $g$ but not $x^n-1$. Let us define
		\begin{equation}\nonumber
			\mathcal{S}:=1-\sum_{i=1}^{u}\frac{1}{p_{i}'}-\sum_{i=1}^{r}\frac{1}{p_{i}}-2\sum_{j=1}^{s}\frac{1}{q^{deg(g_{j})}}, \mathcal{S}>0
		\end{equation} 
		and 
		\begin{equation}\nonumber
			\mathcal{M}:=\frac{r+u+2s-1}{\mathcal{S}}+2.
		\end{equation}
		Then $\mathfrak{M}_{f,a,b}>0$, if we have 
		\begin{equation}\lb{E11}
			q^{\frac{n}{2}-2}>
			(2m+2)W(e')W(e){W(g)}^2\mathcal{M}.
		\end{equation}
	\end{thm}
	\bp
	By using Lemma \ref{L4.1}, we get the following expression
	\begin{equation}\nonumber
		\begin{aligned}
			\mathfrak{M}_{f,a,b}&\geq \sum_{i=1}^{u}\{\mathfrak{M}_{f,a,b}(p_{i}'e',e,g,g)-\theta(p_{i}')\mathfrak{M}_{f,a,b}(e',e,g,g)\}\\&+ \sum_{i=1}^{r}\{\mathfrak{M}_{f,a,b}(e',p_{i}e,g,g)-\theta(p_{i})\mathfrak{M}_{f,a,b}(e',e,g,g)\}\\&+ \sum_{j=1}^{s}\{\mathfrak{M}_{f,a,b}(e',e,g_{j}g,g)-\Theta(g_{j})\mathfrak{M}_{f,a,b}(e',e,g,g)\}\\&+ \sum_{j=1}^{s}\{\mathfrak{M}_{f,a,b}(e',e,g,g_{j}g)-\Theta(g_{j})\mathfrak{M}_{f,a,b}(e',e,g,g)\}+\mathcal{S}\mathfrak{M}_{f,a,b}(e',e,g,g).
		\end{aligned}
	\end{equation}
From (\ref{E9}) and using Lemma \ref{L4.2}, we obtain the following expression
	\begin{equation}\nonumber
		\begin{aligned}
			\mathfrak{M}_{f,a,b}&\geq ({\theta(e')\theta(e){\Theta(g)}^2}/q(q-1))\Bigg[\Bigg(\sum_{i=1}^{u}\theta(p_{i}')+\sum_{i=1}^{r}\theta(p_{i})+2\sum_{j=1}^{s}\Theta(g_{j})\Bigg)\{-(2m+2)\\&\times (q-1)W(e)W(e'){W(g)}^2 q^{n/2+1}\}+\mathcal{S}\{q^{n}-(2m+2)q^{n/2+2}W(e)W(e'){W(g)}^2\}\Bigg],
		\end{aligned}
	\end{equation}
	which implies,
	\begin{equation}\nonumber
		\begin{aligned}
			\mathfrak{M}_{f,a,b}\geq  (\mathcal{S}{\theta(e)\theta(e'){\Theta(g)}^2}/q&(q-1))\Bigg[q^{n}-(2m+2)W(e)W(e'){W(g)}^2q^{n/2+2}\\&\times\Bigg\{\Bigg({\sum_{i=1}^{u}\theta(p_{i}')+\sum_{i=1}^{r}\theta(p_{i})+2\sum_{j=1}^{s}\Theta(g_{j})}\Bigg)\Bigg/{\mathcal{S}}+1\Bigg\}\Bigg].
		\end{aligned}
	\end{equation}
	We note that $\mathcal{S}=\sum_{i=1}^{u}\theta(p_{i}')+\sum_{i=1}^{r}\theta(p_{i})+2\sum_{j=1}^{s}\Theta(g_{j})-(u+r+2s-1)$. Then the above turns into 
	\begin{equation}\nonumber
		\begin{aligned}
			\mathfrak{M}_{f,a,b}\geq (\mathcal{S}{\theta(e)\theta(e'){\Theta(g)}^2}/q(q-1))\{q^{n}-(2m+2)q^{n/2+2}W(e)W(e'){W(g)}^2\mathcal{M} \},
		\end{aligned}
	\end{equation}
	where $\mathcal{M}=\frac{u+r+2s-1}{\mathcal{S}}+2$. In fact, if inequality $(\ref{E11})$ holds, $\mathfrak{M}_{f,a,b}>0$ and so $(q,n)\in\mathcal{S}_{m}$.
	\ep
\section{Evaluations}\lb{L5}
In this section we utilize our results to find out the presence of elements having desired properties. The results that are mentioned earlier apply to the arbitrary finite field $\Field_{q^{n}}$ of arbitrary characteristic. For illustration, we explicitly determine each pair $(q,n)$ belonging to $\mathcal{S}_{3}$, where $q=7^k$ and $n\geq 6$. Here let us split our calculations into two parts. Initially, we identify the exceptions $(q,n)$ for $n\geq 8$, and subsequently, we execute the possible exceptions for $n=6,7$. In this article, SageMath \cite{Sm} serves as the computational tool for all significant calculations. From Theorem \ref{T4.3}, it follows that $(q,n)\in\mathcal{S}_{3}$ if we have
	\begin{equation}\lb{E12}
		q^{\frac{n}{2}-2}>8 ~W(e'){W(e)}{W(g)}^2\mathcal{M}.
	\end{equation}
	Also, by Theorem \ref{T3.1}, $(q,n)\in\mathcal{S}_{3}$ if we have
	\begin{equation}\lb{E13}
		q^{\frac{n}{2}-2}>
		8W(Q)W(q^n-1){W(x^n-1)}^2.
	\end{equation}
 Recall that $Q$ is the largest divisor of $q^n-1$ such that gcd($Q,q-1$)$=1$. Clearly, we have $W(Q)\leq W(q^n-1)$.\\
	\textbf{Part I}: Rewrite $n$ as $n = n'\cdot q^i$; $i\geq 0$, where $q,n'\in\mathbb{N}$ be such that $q$ is a prime power, being co prime to $n'$. Furthermore, assume that $d$ be the order of $q$ modulo $n'$, where gcd$(n',q)=1$. Following [{\cite{RH}}, Theorems $2.45$ and $2.47$], $x^{n'}-1$ can be factorized into the product of irreducible polynomials over $\Field_{q}$ in such a way that degree of each factor must be less than or equal to $d$.
	
	Denote $I_{n'}$ as the cardinality of the set containing the irreducible factors of $x^{n'}-1$ over $\Field_{q}$ such that degree of each factor is less than $d$, and let the ratio $\frac{N_{0}}{n'}$ be denoted by $\pi(q,n')$. Observe that the set containing the irreducible factors of $x^{n}-1$ over $\Field_{q}$ and the set containing irreducible factors of $x^{n'}-1$ over $\Field_{q}$ are of equal cardinality, and this results in $n\pi(q,n)=n'\pi(q,n')$. For further computations, we shall use bounds for $\pi(q,n)$, that is provided in the following lemma.
	\begin{lem}\textbf{({\cite{SS}}, Lemma 6.1, Lemma 7.1)}\lb{L6.1}
		Let $q=7^k$ and $n'>4$ be such that $7\nmid n'$. Let $n_{1}'=gcd(n',q-1)$. Then the following hold:
		\begin{itemize}
			\item[(i)] If $n'=2n_{1}'$, then we have $d=2$ and $\pi(q,n')=1/2$.
			\item[(ii)]If $n'=4n_{1}'$ and $q\equiv 1(mod~ 4)$, then we have $d=4$ and  $\pi(q,n')=3/8$.
			\item[(iii)] If $n'=6n_{1}'$ and $q\equiv 1(mod~ 6)$, then $d=6$ and  $\pi(q,n')=13/36$.
			\item[(iv)] Otherwise, $\pi(q,n')\leq 1/3$.
		\end{itemize}
	\end{lem}
	\begin{lem}\lb{L6.2}
		Suppose that $q=p^k$; $k\in\mathbb{N}$ and $n=n'\cdot q^{i}$; $i\in \mathbb{N}\cup \{0\}$, where gcd($n',q)=1$ in addition with $n'\nmid q-1$. Assume that $d(>2)$ be the order of $q$ (mod $n'$). Moreover, let $e'=Q$, $e=q^n-1$ and $g$ is assumed to be the product of all irreducible factors of $x^{n'}-1$ along with each one have degree less than $d$. Then, following Theorem \ref{T4.3}, we get that $\mathcal{M} <2n'$. 
	\end{lem}
	\begin{proof}
	The proof is omitted here, as it can be derived from [\cite{MASI}, Lemma 10].
	\end{proof}
	Let us find the pairs $(q,n)\in\mathcal{S}_{3}$ for $q=7^k$; $n\geq 8$.
	From now on, we assume that $n\geq 8$ and $n=n'\cdot 7^i$, where gcd$(7,n')=1$. Then we have $W(x^n-1)=W(x^{n'}-1)$.
	\begin{lem}\lb{L6.3}
		Let $q=7$ and  $n=n'\cdot 7^i$, where gcd$(7,n')=1$. Then $(7,n)\in\mathcal{S}_{3}$ for all $n\geq 8$ except for $n=8,9,10,12$ and $18$.
	\end{lem}
	\bp
	Firstly, assume that $n'\nmid q^2-1$. Then we must have $n'\geq 5$ and by Lemma \ref{L6.1}, we have $\pi(7,n')\leq 1/3$ unless $n'=36$, because then $n'=6n_{1}'$ and $\pi(q,n')=13/36$. Let $e'=Q$, $e=7^n-1$ and $g$ is the product of all irreducible factors of $x^{n'}-1$ of degree less than $d$. Then by Lemmas \ref{L2.3}, \ref{L6.2} and inequality (\ref{E13}), $(7,n)\in\mathcal{S}_{3}$ if we have
	\begin{equation}\nonumber
		7^{n/2-2}>8~{\mathcal{C}}^2~7^{2n/r}~2^{2n/3}~2n.
	\end{equation}
Observe that the above inequality holds for $r=9.8$ and $n\geq 446$. For $n\leq 445$, we test inequality (\ref{E13}), and get that $(7,n)\in\mathcal{S}_{3}$ except for $n= 9,10,11,15,18,19,20,\\27,30,32$. However, for the remaining pairs, we choose the values of $e'$, $e$, $g$, $\mathcal{S}$ and $\mathcal{M}$ (see Table \ref{Table3}) such that inequality (\ref{E12}) is satisfied and get that $(7,n)\in\mathcal{S}_{3}$ unless $n=9,10$ and $18$. We now consider $n'=36$ and by testing inequality $7^{36\cdot 7^i/2-2}>576\cdot\mathcal{C}^2\cdot7^{72\cdot 7^i/r}\cdot 2^{13/18}$, we get $(7,36\cdot 7^i)\in\mathcal{S}_{3}$ for $i\geq 1$ and $r=10$. For the sole remaining pair $(7,36)$, we observe that inequality (\ref{E12}) is verified for certain values of $e'$, $e$, $g$, $\mathcal{S}$ and $\mathcal{M}$ (see Table \ref{Table3}).

	Secondly, assume that $n'|q^2-1$. By Lemmas \ref{L2.3}, \ref{L2.4} and inequality (\ref{E13}), $(7,n)\in\mathcal{S}_{3}$ if we have
	\begin{equation}\nonumber
		7^{n'\cdot 7^i/2-2}>8~{\mathcal{C}}^2~7^{2n'\cdot 7^i/r}~2^{2n'}.
	\end{equation}
	Taking $r=9$, the above inequality holds for $i\geq 3$, when $n'=1$, for $i\geq 2$, when $n'=2,3,4,6,8,12$ and for $i\geq 1$, when $n'=16,24,48$. Hence, for $n\geq 8$, $(7,n)\in\mathcal{S}_{3}$ except when $n=8,12,14,16,21,24,28,42,48,49,56,84$. For this exceptions, we test inequality (\ref{E13}) and get $(7,n)\in\mathcal{S}_{3}$ except when $n=8,12,14,16,24,48$. However, for the remaining pairs, we choose the values of $e'$, $e$, $g$, $\mathcal{S}$ and $\mathcal{M}$ (see Table \ref{Table3}) such that inequality (\ref{E12}) is satisfied and get that $(7,n)\in\mathcal{S}_{3}$ unless $n=8$ and $12$.
	\ep
	\begin{lem}
		Let $q=49$ and  $n=n'\cdot 7^i$, where gcd$(7,n')=1$. Then $(49,n)\in\mathcal{D}_{2}$ for all $n\geq 9$.
	\end{lem}
	\bp
	Firstly, assume that $n'\nmid q^2-1$. Then $n'\nmid q-1$ and thus by Lemma \ref{L2.4}, we have $W(x^{n'}-1)\leq 2^{\frac{3}{4}n'}$. Then inequality (\ref{E13}) is true, if $q^{n/2-2}>8~ \mathcal{C}^2 q^{2n/r}~2^{3n/2}$ is true. Choosing $r=10.5$, the latter inequality holds for all $n\geq 381$. For $n\leq 380$, we verify inequality (\ref{E13}) and get that $(49,n)\in\mathcal{S}_{3}$ unless $n=9$ and $18$. However, for the remaining pairs, we choose the values of $e'$, $e$, $g$, $\mathcal{S}$ and $\mathcal{M}$ (see Table \ref{Table3}) such that inequality (\ref{E12}) is satisfied and get that $(49,n)\in\mathcal{S}_{3}$ for all $n$.
	
		\begin{center}
		\begin{table}[h]
			\centering
			\caption{}
			\begin{tabular}{|c|c|c|c|c|c|c|c|}
				\hline  $(q,n)$ & $e'$& $e$ & $g$&$\mathcal{S}$&$\mathcal{M}$\\
				\hline
				$(7,11)$&$1$&$2$&$1$&$0.379164614709749 $&$23.0990152815921$\\
				$(7,14)$&$1$&$2$&$x+1$&$0.291669794015721 $&$36.2853466665835$\\
				$(7,15)$&$1$&$2$&$x^2 + x + 1$&$0.207947594468628$&$78.9424625511299$\\
				$(7,16)$&$1$&$6$&$x^2+6$&$0.194961580806272 $&$109.713529574153$\\
				$(7,19)$&$1$&$1$&$x+6$&$0.126907974235963$&$135.955333400811$\\
				$(7,20)$&$1$&$2$&$x^2+6$&$0.0219001519714673 $&$1006.55923907116$\\
				$(7,24)$&$5$&$30$&$x^6+6$&$0.271667188760882$&$123.472159190509$\\
				$(7,27)$&$1$&$2$&$x^2 + x + 1$&$0.186434908720237$&$130.731256204889$\\
				$(7,30)$&$1$&$2$&$x^6+6$&$0.252361603032526 $&$112.951902601408$\\
				$(7,32)$&$1$&$2$&$x^{16}+6$&$0.138344865742225 $&$146.566261224797$\\
				$(7,36)$&$1$&$6$&$x^6+6$&$0.0815701713798487$&$431.078416876374$\\
				$(7,48)$&$5$&$30$&$x^{24}+6$&$0.0315593546237637$&$1427.88467148551$\\
				\hline
				$(7^2,9)$&$1$&$2$&$1$&$0.336456330954422$&$61.4430782243456$\\
				$(7^2,10)$&$1$&$2$&$x+1$&$0.0219001519714673$&$1006.55923907116$\\
				$(7^2,12)$&$5$&$30$&$x+1$&$0.190034535699657$&$196.701451837560$\\
				$(7^2,15)$&$1$&$2$&$1$&$0.129912623440689$&$263.714366930035$\\
				$(7^2,16)$&$5$&$30$&$x+1$&$0.262765246139282$&$150.421454408501$\\
				$(7^2,18)$&$1$&$6$&$x+1$&$0.0232271340399648$&$1508.85831234188$\\
				$(7^2,20)$&$1$&$2$&$x^4 + 6$&$0.00893876673760447$&$3805.65670098153$\\
				$(7^2,24)$&$902785$&$5416710$&$x+1$&$0.0058477612584259$&$10091.3311803707$\\
				$(7^2,30)$&$55$&$330$&$x+1$&$0.353135978712364$&$169.074451646448$\\
				$(7^2,48)$&$5$&$30$&$\frac{x^{48}-1}{x^{12}-1}$&$0.00527505642356318 $&$10428.4287590025$\\
				\hline
				$(7^3,8)$&$1$&$6$&$1$&$0.279932899745072$&$94.8794008266895$\\
				$(7^3,9)$&$1$&$2$&$1$&$0.483964545975008$&$66.0542788884391$\\
				$(7^3,10)$&$1$&$2$&$1$&$0.298329291645090$&$79.0960165298223$\\
				$(7^3,12)$&$1$&$6$&$1$&$0.244971922619374$&$140.791415915968$\\
				$(7^3,18)$&$1$&$114$&$x+1$&$0.776753550747086$&$67.6578910401477$\\
				\hline
				$(7^4,8)$&$1$&$2$&$1$&$0.335012920719318 $&$82.5939064738975$\\
				$(7^4,9)$&$1$&$2$&$1$&$0.0915444731815404$&$296.938613568257$\\
				$(7^4,10)$&$1$&$2$&$1$&$0.207272794226151$&$180.508714267778$\\
				$(7^4,12)$&$1$&$30$&$x+1$&$0.512192424178116$&$85.9528231386856$\\
				$(7^4,15)$&$1$&$30$&$x+3$&$0.373733457035790$&$149.163704411760$\\
				\hline
				$(7^5,8)$&$1$&$2$&$1$&$0.0157216548212150$&$1719.37646622071$\\
				\hline
				$(7^6,8)$&$1$&$6$&$1$&$0.197106668135930$&$189.715617893169$\\
				\hline
			\end{tabular}
			\label{Table3}
		\end{table}
	\end{center}
Secondly, assume that $n'|q^2-1$.  By Lemmas \ref{L2.3}, \ref{L2.4} and inequality (\ref{E13}), $(49,n)\in\mathcal{S}_{3}$ if we have
	\begin{equation}\nonumber
		49^{n'\cdot 7^i/2-2}>8~{\mathcal{C}}^2~49^{2n'\cdot 7^i/r}~2^{2n'}.
	\end{equation}Taking $r=9$, the above inequality holds for $i\geq 2$, when $n'=1, 2, 3, 4, 5$ and for $i\geq 1$, otherwise. Thus for $n\geq 8$, $(49,n)\in\mathcal{S}_{3}$ except when $n'=$ 8, 10, 12, 14, 15, 16, 20, 21, 24, 25, 28, 30, 32, 35, 40, 48, 50, 60, 75, 80, 96, 100, 120, 150, 160, 200, 240, 300, 400, 480, 600, 800, 1200, 2400. For each of the values of $n$, we test inequality (\ref{E13}) and get that $(49,n)\in\mathcal{S}_{3}$ unless $n=8, 10, 12, 15, 16, 20, 24, 30, 48$. However, for the remaining pairs, we verify inequality (\ref{E12}) and find the values of $e'$, $e$, $g$, $\mathcal{S}$ and $\mathcal{M}$ (listed in Table \ref{Table3}). Thus we get $(49,n)\in\mathcal{S}_{3}$ unless $n=8$.
	\ep
	\begin{lem}
		Let $k\in\mathbb{N}$ and $q=7^k$. Then $(q,n)\in\mathcal{S}_{3}$ for $n\geq 8$ and $k\geq 3$.
	\end{lem}
	\bp
	From Lemma \ref{L2.3}, $W(q^n-1)<C\cdot (q^n-1)^{1/r}$ for some positive real number $r$ and also we have $W(x^n-1)\leq 2^n$. Therefore, by inequality (\ref{E13}), $(q,n)\in\mathcal{S}_{3}$ if
	\begin{equation}\lb{E14}
		q^{n/2-2}>8\cdot \mathcal{C}^2\cdot q^{2n/r}\cdot 2^{2n}.
	\end{equation}
	For $r=9.5$, Lemma \ref{L2.3} gives $\mathcal{C}<1.46\times 10^7$ and thus inequality (\ref{E14}) holds for $n\geq 8$ and $k\geq 76$. For each $3\leq k\leq 75$ and for proper choice of `$r$', we get $n_{k}$'s such that for all $n\geq n_{k}$, inequality (\ref{E14}) is satisfied, that are listed in Table \ref{Table1}.  
		\begin{center}
		\begin{table}[h]
			\centering
			\caption{}
			\begin{tabular}{|c|c|c|}
				\hline $r$ & $k$ & $n_k$ \\
				\hline 10 & \{3\} & 152 \\
				9.0 & \{4\} & 57\\
				8.5 & \{5\} & 36\\
				8.5 & \{6\} & 28 \\
				8.5 & \{7\} & 23 \\
				8.5 & \{8\} & 20 \\
				8.5 & \{9\} & 18 \\
				9 & \{10\} & 17 \\
				9 & \{11\} & 16 \\
				9 & \{12\} & 15 \\
				9 & \{13,14\} & 14 \\
				9 & \{15,16,17\} & 13 \\
				9 & \{18,19,20,21\} & 12 \\
				9 & \{22,23,\ldots,27\} & 11 \\
				9 & \{28,29,\ldots,40\} & 10 \\
				9.5 & \{41,42,\ldots,75\} & 9 \\
				\hline
			\end{tabular}
			\label{Table1}
		\end{table}
	\end{center}
We calculate the values of $W(Q)$, $W(q^n-1)$ and $W(x^n-1)$ precisely for each of the above values of $k$ and $n$ (mentioned in Table \ref{Table1}) and verify inequality
	\begin{equation}\nonumber
		q^{\frac{n}{2}-2}>
		8\cdot W(Q){W(q^n-1)}{W(x^n-1)}^2.
	\end{equation}	
	Consequently, we get $(q,n)\in\mathcal{S}_{3}$ unless $(q,n)$ equals $(7^3,8)$, $(7^3,9)$, $(7^3,10)$, $(7^3,12)$, $(7^3,18)$, $(7^4,8)$, $(7^4,9)$, $(7^4,10)$, $(7^4,12)$, $(7^4,15)$, $(7^5,8)$, $(7^6,8)$. For these exceptions, we find the values of $e'$, $e$, $g$, $\mathcal{S}$ and $\mathcal{M}$ for all $(q,n)$ (listed in Table \ref{Table3}) such that inequality (\ref{E12}) is satisfied. Thus, $(7^k,n)\in\mathcal{S}_{3}$ for all $n\geq 8$ and $k\geq 3$.
	\ep
\textbf{Part II}: In this part, we execute computations for $n=6,7$. The following lemma will be utilized in this part for computation.
	\begin{lem}\lb{L5.6}
		Let $M\in\mathbb{N}$ such that $\om(M)\geq 2828$. Then we have $W(M)<M^{1/13}$.
	\end{lem}
	\begin{proof}Let $S=\{2,3,\ldots,25673\}$ be the set containing first $2828$ primes. Clearly the product of all elements in $S$ surpasses $2.24\times 10^{11067}$. Let us decompose $M$ as the product of two co prime positive integers $M_{1}$ and $M_{2}$ such that prime divisors of $M_{1}$ come from the least $2828$ prime divisors of $M$ and remaining prime divisors are divisors of $M_{2}$. Therefore, $M_{1}^{1/13}>2.16\times10^{851}$ , on the other hand $W(M_{1})<2.06\times 10^{851}$. Hence we draw the conclusion, as  $p^{1/13}>2$ for any prime $p>25673$. 
	\end{proof}
	First we assume that $\om(q^n-1)\geq 2828$. Note that $W(x^n-1)\leq 2^7$ and thus by the \ref{L5.6}, $(q,n)\in\mathcal{S}_{3}$ if we have $q^{\frac{n}{2}-2}>8\cdot q^{\frac{2n}{13}}\cdot 2^{14}$, that is, if $q^n>2^{\frac{442n}{9n-52}}$ then $(q,n)\in\mathcal{S}_{3}$. But $n\geq 6$ gives $\frac{442n}{9n-52}\leq 1326$. Hence, if $q^{n}>2^{1326}$ then $(q,n)\in\mathcal{S}_{3}$, which is valid when $ \om(q^n-1)\geq 2828$. Let us suppose that $88\leq \om(q^n-1)\leq 2827$. In Theorem \ref{T4.3}, choose $g=x^n-1$ and $e$ (or $e'$) is assumed to be the product of least $88$ prime divisors of $q^n-1$ ( or $Q$), that is, $W(e)=W(e')=2^{88}$, then $r\leq 2739$ and $\mathcal{S}$ assumes its minimum positive value when $\{p_{1},p_{2},\ldots,p_{2739}\}=\{461,463,\ldots,25667\}$. This gives $\mathcal{S}>0.0044306$ and $\mathcal{M} <1.24\times 10^6$. Thus $8\mathcal{M} W(e')W(e){W(g)}^2<1.5518994\times 10^{64}=R$(say). By Sieve variation, we get $(q,n)\in\mathcal{S}_{3}$ if we have $q^{\frac{n}{2}-2}>R$, that is, if $q^n>R^{\frac{2n}{n-4}}$. Since $n\geq 6$ implies $\frac{2n}{n-4}\leq 6$, we have $(q,n)\in\mathcal{S}_{3}$ if $q^n>1.396951\times 10^{385}$. Hence, $\om(q^n-1)\geq 158$ gives $(q,n)\in\mathcal{S}_{3}$.
	\begin{center}
		\begin{table}[h]
			\centering
			\caption{}
			\begin{tabular}{|c|c|c|c|c|}
				\hline  $a\leq \om(q^n-1)\leq b$ & $W(e)/W(e')$ & $\mathcal{S}>$& $\mathcal{M}<$& $8\mathcal{M} W(e'){W(e)}{W(g)}^2<$\\
				\hline
				
				  $a=17, b=157$ & $2^{17}$ & $0.02162406$ &$12904.293824$&$2.90579\times 10^{19}$ \\
				$a=10, b=60$ & $2^{10}$ &$0.0550598$& $1800.044933$& $2.47397\times 10^{14}$\\
				$a=8, b=47$ & $2^{8}$ &$0.00340868$& $22591.376714$& $1.94059\times10^{14}$\\
				\hline
			\end{tabular}
			\label{Table2}
		\end{table}
	\end{center}
	We repeat the steps given in Theorem \ref{T4.3} by using the data given in the second column of Table \ref{Table2}. Therefore we get, $(q,n)\in\mathcal{S}_{3}$ if we have $q^{\frac{n}{2}-2}>1.94059\times10^{14}$. This gives the scenarios that  $n=6, q>1.94059\times10^{14}$; $n=7, q>(1.94059\times10^{14})^{{2}/{3}}$. Thus, the only possible exceptions are $(7,6),(7^2,6),\ldots,(7^{16},6)$; $(7,7),(7^2,7),\ldots,(7^{11},7)$.  However Table \ref{Table6} implies that Theorem \ref{T4.3} holds for $(7^5,6),(7^6,6),\ldots,(7^{16},6)$;$(7^3,7),(7^,7),\ldots,(7^{11},7)$. Thus the only possible exceptions are $(7,6),(7^2,6),(7^3,6)$ and $(7,7)$.
	\begin{center}
		\begin{table}[h]
			\centering
			\caption{}
			\begin{tabular}{|c|c|c|c|c|c|c|c|}
				\hline  $(q,n)$ & $e'$ & $e$& $g$&$\mathcal{S}$&$\mathcal{M}$\\
				\hline
				$(7^4,6)$&$1$&$6$&$1$&$0.434016210002031$&$66.5137194296705$\\
				$(7^5,6)$&$1$&$2$&$1$&$0.257002547699352$&$107.057324301646$\\
				$(7^6,6)$&$1$&$6$&$1$&$0.303162160154874$&$91.0612469122359$\\
				$(7^7,6)$&$1$&$2$&$1$&$0.460243100976110$&$62.8374138376349$\\
				$(7^8,6)$&$1$&$6$&$1$&$0.322185830859915$&$104.425360891641$\\
				$(7^9,6)$&$1$&$1$&$1$&$0.0178814663260537$&$1679.71476080177$\\
				$(7^{10},6)$&$1$&$6$&$1$&$0.0818969304893583 $&$465.997854045796$\\
				$(7^{11},6)$&$1$&$2$&$1$&$0.396638261656470 $&$82.6780462035085$\\
				$(7^{12},6)$&$1$&$2$&$1$&$0.208001805090034$&$189.498372829595$\\
				$(7^{13},6)$&$1$&$2$&$1$&$0.451835063251708$&$75.0355005264751$\\
				$(7^{14},6)$&$1$&$6$&$1$&$0.238642380092351$&$148.662969026941$\\
				$(7^{15},6)$&$1$&$2$&$1$&$0.282943029313370$&$150.439776381568$\\
				$(7^{16},6)$&$1$&$2$&$1$&$0.0213997081969370$&$1964.64358436493$\\
				\hline
				$(7^{2},7)$&$1$&$2$&$1$&$0.536567753199395$&$20.6369753686705$\\
				$(7^3,7)$&$1$&$1$&$1$&$0.0388161015503813$&$259.625047353622$\\
				$(7^4,7)$&$1$&$2$&$1$&$0.376551093326836$&$36.5238673592599$\\
				$(7^5,7)$&$1$&$1$&$1$&$0.0968025701017207$&$125.963650834790$\\
				$(7^6,7)$&$1$&$1$&$1$&$0.00143402424023756$&$13251.4273575547$\\
				$(7^7,7)$&$1$&$1$&$1$&$0.131401454426554$&$100.933455925073$\\
				$(7^8,7)$&$1$&$2$&$1$&$0.369673098692035$&$61.5120393608290$\\
				$(7^9,7)$&$1$&$1$&$1$&$0.0166791949792349$&$841.369047332896$\\
				$(7^{10},7)$&$1$&$2$&$1$&$0.447962587202473  $&$57.8082320135819$\\
				$(7^{11},7)$&$1$&$1$&$1$&$0.0963847008090025$&$157.626358479073$\\
				\hline
			\end{tabular}
			\label{Table6}
		\end{table}
	\end{center}
	The above discussions leads us to conclude the following.
	\begin{thm}
		Let $q,k,n\in\mathbb{N}$ such that $q=7^k$ and $n\geq 6$. The $(q,n)\in\mathcal{S}_{3}$ unless the following possible exceptions:
		\begin{itemize}
			\item[1.] $q=7,7^2,7^3~\text{and}~ n=6$;
			\item[2.] $q=7,7^2~\text{and}~ n=8$;
			\item[3.] $q=7~\text{and}~n=7,9,10,12$ and $18$.
		\end{itemize}
	\end{thm}
\section{Acknowledgments}
We sincerely appreciate and acknowledge the reviewers for their helpful comments and suggestions. First author is supported by the National Board for Higher Mathematics (NBHM), Department of Atomic Energy (DAE), Government of India, Ref No.  0203/6/2020-R\&D-II/7387.
\section{Statements and Declarations}
There are no known competing financial interests or personal relationships that could have influenced this paper's findings. All authors are equally contributed.

\end{document}